\documentclass[12pt,psamsfonts]{amsart}
\usepackage{amssymb,amsmath,amscd,eucal,epsfig}
\usepackage{tikz}

 \def\Dj{\hbox{D\kern-.73em\raise.30ex\hbox{-} \raise-.30ex\hbox{}}}
 \def\dj{\hbox{d\kern-.33em\raise.80ex\hbox{-} \raise-.80ex\hbox{\kern-.40em}}}

\def\id{\operatorname{id}}
\def\im{\operatorname{im}}
\def\ad{\operatorname{ad}}
\def\tr{\operatorname{tr}}
\def\Ric{\operatorname{Ric}}
\def\ric{\operatorname{ric}}        %added
\def\rk{\operatorname{rk}}          %added
\def\Span{\operatorname{Span}}      %added
\def\es{\emptyset}
\def\sq{\subseteq}
\def\<{\langle}                     %added
\def\>{\rangle}                     %added
\def\ip{\<\cdot,\cdot\>}            %added

\def\N{\mathbb N}
\def\Z{\mathbb Z}
\def\R{\mathbb R}
\def\Q{\mathbb Q}
\newtheorem{thm}{Theorem}[section]
\newtheorem{prop}[thm]{Proposition}
\newtheorem{lem}[thm]{Lemma}
\newtheorem{cor}[thm]{Corollary}
\newtheorem{rem}[thm]{Remark}
\newcommand{\ben}{\begin{enumerate}}
\newcommand{\een}{\end{enumerate}}

\renewcommand{\theenumi}{\alph{enumi}}
\renewcommand{\labelenumi}{(\theenumi)}
\renewcommand{\theenumii}{\roman{enumii}}

\theoremstyle{plain}

\theoremstyle{definition}

\numberwithin{equation}{section}

\begin{document}
\setcounter{page}{1}

%\begin{center}

\title[Labeling the Product of Paths and Cycles]{Distance two labeling of Direct Product of Paths and Cycles}

%\end{center}

\author[D.O. Ajayi]{Deborah Olayide Ajayi$^1$}
\address{$^1$Department of Mathematics,
\newline \indent University of Ibadan,
\newline \indent Ibadan,
\newline \indent Nigeria}
\email{olayide.ajayi@mail.ui.edu.ng; adelaideajayi@yahoo.com}

\author[T.C Adefokun]{Tayo Charles Adefokun$^2$ }
\address{$^2$Department of Computer and Mathematical Sciences,
\newline \indent Crawford University,
\newline \indent Nigeria}
\email{tayoadefokun@crawforduniversity.edu.ng; tayo.adefokun@gmail.com}


\begin{abstract} Suppose that $[n]=\left\{0,1,2,...,n\right\}$ is a set of non-negative integers and $h,k \in [n]$. The $L(h,k)$-labeling of graph $G$ is the function
$l:V(G)\rightarrow[n]$ such that $\left|l(u)-l(v)\right|\geq h$ if the distance $d(u,v)$ between $u$ and $v$ is one and $\left|l(u)-l(v)\right| \geq k$ if the distance $d(u,v)$ is two. Let $L(V(G))=\left\{l(v): v \in V(G)\right\}$ and let $p$ be the maximum value of $L(V(G)).$ Then $p$ is called $\lambda_h^k-$number of $G$ if $p$ is the least possible member of $[n]$ such that $G$ maintains an $L(h,k)-$labeling. In this paper, we establish $\lambda_1^1-$ numbers of $P _m \times C_n$ graphs for all $m \geq 2$ and $n\geq 3$.
\end{abstract}

\maketitle

%%%%%%%%%%%%%%%%%%%%%%%%%%%%%%%%%%%%%%%%%%%%%%%%%%%%%%%%%%%%%%%%%%%%%%
\section{Introduction}
%%%%%%%%%%%%%%%%%%%%%%%%%%%%%%%%%%%%%%%%%%%%%%%%%%%%%%%%%%%%%%%%%%%%%%
Let $l:V(G)\rightarrow [n]=$ $\left\{0,1,2,\cdots, n\right\}$ be  a non negative function on the vertex set $V(G)$ of a graph $G$. Given any two fixed non-negative integers $h,k$, the $L(h,k)$-labeling (or distance two labeling) of $G$ is defined such that for any edge $uv \in E(G)$, $\left|l(u)-l(v)\right| \geq h$ and if the distance $d(u,v)$ is two for  $ u,v \in V(G),$ then $\left|l(u)-l(v)\right|\geq k.$ The aim of $L(h,k)-$labeling is to obtain the smallest non negative integer $\lambda_h^k(G),$ such that there exists an $L(h,k)$-labeling of $G$ with no $l(v) \in L(V(G))$ greater than $\lambda_h^k(G)$, where $L(V(G))$ is the set of all labels on $V(G)$.

In \cite{GY10}, Griggs and Yeh showed that any graph $G$ with maximum degree $\Delta > 1$ has $\lambda_2^1(G)\leq \Delta^2+2\Delta$ and went further to put forward a conjecture that $\lambda_2^1(G) \leq \Delta ^2.$ Chang and Kuo, in \cite{CK4} improved on Griggs and Yeh's bound by showing that $\lambda_2^1(G) \leq \Delta(\Delta+1) $, Kral' and Skrekovski \cite{KS14} went another step showing that $\lambda_2^1(G) \leq \Delta(\Delta+1)-1$ while Goncalves in \cite{G8} proved that $\lambda_2^1 (G) \leq \Delta(\Delta+1)-2$. The interest in the Griggs-Yeh conjecture and in improving on the existing bounds have inspired a lot of work in the direction of $L(h,k)$-labeling, mostly on $h=2$, $k=1$. (See \cite{CK4}\cite{CKKL5}\cite{GM7}\cite{GKM9}\cite{S16}.) (An extensive review of all known results on $L(h,k)-$labeling can be seen in \cite{C2}.) It is obvious that $L(2,1)-$labeling is an $L(1,1)-$labeling, therefore results on $L(2,1)$-labeling provide upper bound for $L(1,1)$-labeling of graphs and $$\lambda_2^1(G)+1 \geq \lambda_1^1(G)+1=\lambda(G^2)$$ where $\lambda(G^2)$ is the chromatic number of the square of $G$.

 Suppose that $G$ and $H$ are graphs. The Cartesian product and the direct product of $G$ and $H$, $G\square H$ and $G\times H$ respectively, have vertex set $V(G) \times V(H)$, while the edge sets are
 $E(G\square H)$= $\left\{((x_1,x_2),(y_1,y_2)): (x_1,y_1) \in E(G) \; \textrm {and} \; x_2=y_2 \; \textrm {or} \; \right. \\ \left.(x_2,y_2) \in E(H)  \; \textrm {and} \; x_1=y_1\right\} $ and
 $E(G \times H)= \left\{((x_1, x_2), (y_1,y_2)):(x_1,y_1)\in E(G)\;\right. \\ \left.\textrm {and} \;  (x_2,y_2) \in E(H)\right\}$ respectively.

The $L(h,k)-$labeling of the Cartesian product $G\square H$ has been extensively investigated with $\lambda_h^k(G\square H)$ obtained for various types of graphs $G$ and $H$, while numerous upper and lower bounds have been suggested (see \cite{GMS6}\cite{CY8}\cite{KY15}\cite{SS17}\cite{WP19}\cite{WGM20}). Most of the work on $L(h,k)$ labeling consider $h=2$ and $k=1$; although Chiang and Yan in \cite{CY8} and Georges and Mauro in \cite{GM7} worked on the $L(1,1)$labeling of Cartesian products of paths and cycles and Sopena and Wu in \cite{SW18} worked on Cartesisn products of cycles. In case of direct product graphs, Jha et al \cite{JKV13}, established $\lambda_2^1(C_m \times C_n)$ for some values of $m$ and $n$.
 %However, not much work is known to have been done on the %$L(h,k)$-labeling of direct product of graphs. In \cite{CPP3}, %Calamoneri et al. established upper bound values for  %$L(2,1)-$labeling of some direct product graphs and in \cite{JKV13}, %some values of $\lambda_2^1(C_m \times C_n)$ with some conditions on %$m$ and $n$ were introduced.

In this paper, we determine $\lambda_1^1(P_m \times P_n)$ and $\lambda_1^1(P_m \times C_n)$ where $P_m$ and $P_n$ are paths of length $m-1$ and $n-1$ respectively and $C_n$ is a cycle of length $n$ for all $m,n \geq 2$. We also deduce $\lambda_1^1(C_m \times C_n)$ for $m,n\equiv 0\mod5$. Thus, we extend the results in \cite{GM7} and \cite{CY8} to direct product graphs among other results.

\section{Preliminaries}
The following results and definitions are necessary.\\
Let $m$ be a non-negative integer. $P_m=u_0u_1u_2...u_{m-1}$ is a path of length $m-1$, where $u_i\in V(P_m),$ for all $i\in[m-1]$;
$C_m=u_0u_1u_2...u_{m-1}u_0$ is a cycle of length $m$, where $u_i\in V(C_m),$ for all $i\in[m-1]$. Let $v\in V(G)$, we denote by $l(v)$ the label on $v$ and let $U \subseteq V(G)$. Then $L(U)$ is a set of labels on $U$.

Suppose $P_m \times P_n$ is a direct product paths and $G'$ is a component of $P_m \times P_n$. Then \\
$U_j=\left\{u_iv_j\right\}\subset V(G')$, for some $j\in [n-1]$, and for all $i\in [(m-1)(\epsilon)]$ or for all $i \in [(m-1)(o)].$
\\
$V_i=\left\{u_iv_j\right\}\subset V(G')$, for some $i\in [m-1]$, and for all $j\in [(n-1)(\epsilon)]$ or for all $j \in [(n-1)(o)].$
\begin{thm} \label{a}\cite{IK12} Graph $G\times H$ is connected if and only if $G$ and $H$ are connected and at least one of $G$ and $H$ is non-bipartite.
\end{thm}
\begin{rem} \label{b} \begin{itemize} \item[(i)] \rm {Since $P_m$ is bipartite for all $m \geq 2$, then for $P_m \times P_n $, there exist $G_1 \subset P_m \times P_n$ and $G_2 \subset P_m \times P_n$ such that $G_1$ and $G_2$ are components of $P_m \times P_n$}.
%\end{rem}
%\begin{rem}
\item[(ii)] \rm{From Theorem \ref{a} and the Remark above, it is clear that $P_m \times P_n$ is not a connected graph. Suppose $P_m=u_0u_1u_2...u_{m-1}$ and $P_n=v_0v_1v_2...v_{n-1},$ then \\
$V(G_1)=\left\{u_i,v_j: i \in [(m-1)(\epsilon)], j \in [(n-1)(\epsilon)] \; \textrm {or} \; i \in [(m-1)(o)]; j\in (n-1)[o]\right\}$\\
 $V(G_2)=\left\{u_i,v_j: i \in [(m-1)(\epsilon)], j \in [(n-1)(o)] \; \textrm {or} \; i \in [(m-1)(o)]; j\in (n-1)[\epsilon]\right\}.$}
% \end{rem}
%\begin{rem}\label{c}
\item[(iii)] \rm{Suppose $G$ is a graph such that $G=G'\cup G''$, where $G', G''$ are components of G, then, $\lambda_1^1(G)=max\left\{\lambda_1^1(G'),\lambda_1^1(G'')\right\}.$}
%\end{rem}
%\begin{rem}\label{d}
\item[(iv)] \rm{For a direct product graph, $P_m \times P_2$, $m \geq 2,$ its components $G_1$ and $G_2$ are paths $P'_m$ and $P''_m$ respectively such that\\ $P'_m= u_0v_0u_1v_1u_2v_0...u_{m-1}v_1(u_{m-1}v_0)$ (if $m$ is even) and\\ $P''_m=u_0v_1u_1v_0u_2v_1...u_{m-1}v_0(u_{m-1}v_1)$ (if $m$ is odd)}.
\end{itemize}
\end{rem}

The following are known results for $L(1,1)$-labeling of paths, cycles and $L(h,k)$-labeling of stars, $k \leq h$.

\begin{lem}\label{e} \cite{BB1} Let $P_m$ be a path of length $m-1$. $\lambda_1^1(P_m)=1,$ for $m=2$ and $\lambda_1^1(P_m)=2$ \rm for all $m \geq 3.$
\end{lem}
\begin{lem} \label{f} \cite{BB1} Let $C_m$ be cycle of length $m$. Then $\lambda_1^1(C_m)=2$  for $m \equiv 0 \mod 3$ and $\lambda_1^1(C_m)=3$ for $m \; \not\equiv 0 \mod 3.$
\end{lem}
%The following result presents a general $\lambda_h^k$-value for stars for $k \leq h$.
\begin{lem} \label{g} \cite{CPP3} Let $K_{1, \Delta}$ be a star of order $\Delta +1.$ Then, $\lambda_h^k(K_{1,\Delta})=(\Delta-1)k+h$ if $h \geq k.$
\end{lem}

%Next we define some of the notions we use in this work.
% $P_m=u_0u_1u_2...u_{m-1}$ is a path of length $m-1$, where $u_i\in V(P_m),$ for all $i\in[m-1]$; \\
%$C_m=u_0u_1u_2...u_{m-1}u_0$ is a cycle of length $m$, where $u_i\in V(C_m),$ for all $i\in[m-1]$.

%Suppose $P_m \times P_n$ is a direct product graph and $G'$ is a component of $P_m \times P_n$ Then \\
%$U_j=\left\{u_iv_j\right\}\subset V(G')$, for some $j\in [n-1]$, and for all $i\in [(m-1)(\epsilon)]$ or for all $i \in [(m-1)(o)].$
%\\
%$V_i=\left\{u_iv_j\right\}\subset V(G')$, for some $i\in [m-1]$, and for all $j\in [(n-1)(\epsilon)]$ or for all $j \in [(n-1)(o)].$

Henceforth we refer to direct product graph as \textit {product graph}.

 \section{$L(1,1)$-Labeling of $P_m \times P_n$}
\begin{prop} \label{h} $\lambda^1_1(P_2 \times P_2)=1.$
\end{prop}

\begin{proof}Clearly, $G$ consists of connected components $P'_2$ and $P''_2$. By Lemma \ref{e}, $\lambda_1^1(P'_2)=\lambda_1^1(P''_2)=1.$
\end{proof}
We extend the graph in Theorem \ref{h} to $m \geq 3.$
\begin{prop} \label{i} For $m \geq 3$, $\lambda_1^1(P_m \times P_2)=2$.
\end{prop}
\begin{proof}$P_m \times P_2$ consists of two connected components $P'_m$ and $P''_m$. By Lemma \ref{e}, $\lambda_1^1(P'_m)=\lambda_1^1(P''_m)=2$ and the result follows from Remark \ref{b}(iii).
\end{proof}

The next results establish $\lambda_1^1(P_m \times P_n), \; m,n \geq 3$.
\begin{lem}\label{j} Let $u_iv_j \in P_m \times P_n, \; n,m \geq 3,$
%such that $u_i \in V(P_m)$ and $u_j \in V(P_n).$
Suppose $d_{u_i}=d_{v_j}=2$ then $d_{u_iv_j}=4.$
\end{lem}
\begin{proof}Let $u_{i-1}u_iu_{i+1}=P'_3, \; P'_3 \subseteq P_m,\; m \geq 3$ and let $v_{j-1}v_jv_{j+1}=P''_3, \; P''_3 \subseteq P_n,\; n \geq 3$. By the definition of direct product of graphs, \ \\
$V(P'_m \times P''_n)$=
$\left\{u_{i-1}v_{j-1},u_{i-1}v_j,u_{i-1}v_{j+1},u_iv_{j-1},u_iv_j,u_iv_{j+1},u_{i+1}v_{j-1},u_{1+1}v_j,u_{i+1}v_{j+1}\right\}$ \\
 $\subseteq V(P_m \times P_n)$. Since $d_{u_i}=d_{v_j}=2,$ then by the definition of direct product of graphs,  $u_iv_j \in V(P'_3 \times P''_3)$ is adjacent to all the members of
 $\left\{u_{i-1}v_{j-1},u_{i+1}v_{j-1},u_{i+1}v_{j+1},u_{i-1}v_{j+1}\right\}.$  Thus, $d_{u_iv_j}=4$.
\end{proof}

\begin{prop}\label{l}  Suppose $m,n \geq 3$. Then $\lambda_1^1(P_m \times P_n)=4$ for all $m,n \geq 3.$
\end{prop}
\begin{proof} Let $G_1$ be a connected component of $P_m \times P_n$. By Lemma \ref{j}, there exists a star $K_{1,4} \subseteq G_1$. By Lemma \ref{g}, $\lambda_1^1(K_{1,4})=4$ and thus, $\lambda_1^1(P_m \times P_n)\geq 4$. Let $u_iv_j \in V(P_m \times P_n)$. For all $u_iv_j \in V(P_m \times P_n)$, $l(u_iv_j)= \left\lfloor \frac{i+3j}{2}\right\rfloor\; \mod \;5$. Thus $\lambda_1^1(P_m \times P_n)\leq 4$ and then the equality follows.
\end{proof}
\begin{rem} \label{n} \rm{By using $l(u_iv_j)= \left\lfloor \frac{i+3j}{2}\right\rfloor \mod 5$ as in the proof of Proposition \ref{l}, given both connected components of $P_m \times P_n$, for all $i\in [m(\epsilon)], then l(u_iv_{10})=l(u_iv_0).$ Furthermore, for all $u_iv_1\in U_1$, $i \in \left\{3,5,7\right\}$ $l(u_iv_1)\notin L(u_{i-2}v_9, u_iv_9, u_{i+2}v_9),$   $\left\{u_{i-2}v_9, u_iv_9, u_{i+2}v_9\right\}\subset U_9$.  We also notice that $l(u_1v_1)$ $\notin L(u_1v_9,u_3v_9,u_9v_9)$, while $l(u_9v_1) \notin L(u_1v_9,u_7v_9,u_9v_9)$. Also, for all $u_1v_j \in V_1, j\in\left\{3,5,7\right\}$, $l(u_1v_j) \notin L(u_9v_{j-2},u_9v_j,u_9v_{j+2})$, $\left\{u_9v_{j-2},u_9v_j,u_9v_{j+2}\right\}\subset V_9$ and $l(u_1v_1) \notin L(u_9v_1,u_9v_3,u_9v_9)$, while $l(u_1v_9) \notin L(u_9v_1,u_9v_7,u_9v_9)$.}
\end{rem} \ \\
The implication of Remark \ref{n} is expressed in the following results.
\begin{cor} Let $C_m$ be a cycle of length $m$, then, $\lambda^1_1(C_{10} \times C_{10})=4.$
\end{cor}
\begin{cor} \label{pp1} For all $m,n \equiv 0 \mod 5$, $\lambda_1^1(C_m \times C_n)=4.$
\end{cor}

\section{$L(1,1)$-Labeling of $P_m \times C_m$}
\begin{lem}\label{q} Let $G=P_m \times P_n$, where $n\geq 4$. Suppose that $\alpha_k \in [4]$, such that for some $v_i \in V(G),$ $l(v_i)= \alpha_k$, $v_j \in V(G)$ is the closest vertex in $V(G)$ to $v_i, \; i \neq j$ such that $l(v_j)= \alpha_k.$ Then $3\leq d(v_i,v_j)\leq 4.$
\end{lem}

\begin{proof} That $3\leq d(v_i,v_j)$ follows directly from the definition of $L(1,1)$-labeling. Next, we show that $d(v_i,v_j)\leq 4.$ Let $S_n$ be a star of order $n+1$. Clearly, $diam(S_n)=2$. Now, suppose that for two stars $S'_4 \subset G$ and $S''_4 \subset G$, there exits some vertex $u_i$ such that $u_i\in V(S'_n)$ and also $u_i\in V(S''_n)$, making $S'_n$ and $S''_n$ to be neighbors. Then, $diam(H)=4$, where $S'_4 \cup S''_4=H\subset G$. Now, suppose $d(v_i,v_j)>4$. Let $ v_i \in V(S'_4)$ such that $l(v_i)=\alpha_k$. Also, let $L(S'_4)=[4]$. Then, $\alpha_k \neq l(v_k)$ for all $v \in V(S''_4)$ since $d(u_i,v_j) \geq 4$. Thus, there exits some $\alpha_j \notin [4]$ such that $\alpha_j \in L(S''_4)$.
Then, $\lambda_1^1(H)\geq 5,$ and consequently, $\lambda_1^1(G) \geq 5$. This is a contradiction.
\end{proof}

\begin{lem}\label{r} Let $v_i,v_j \in V(G)$ be two center vertices of stars $S'_4, S''_4 \subset G$ respectively, and that $d(v_i,v_j)=4$ if $\alpha_i=l(v_i)$ and $\alpha_j=l(v_j)$, $\alpha_i,\alpha_j \in [4]$, then $\alpha_i \neq \alpha_j$.
\end{lem}

\begin{proof}Suppose on the contrary that $v_i,v_j$ are respective centers of $S'_4, S''_4$ such that $d(v_i,v_j)=4$ and $\alpha_i=\alpha_j$. There exists a star $S'''_4\subset G$ with $V(S'''_4)={u_qv_r,u_{q+2}v_r,u_{q+1}v_{r+1},u_qv_{r+2},u_{q+2}v_{r+2}}$, where $0 \leq q,\; q+2\leq m$ and $r \leq 2, r+2 \leq n-3$, such that $v_i=u_{q+1}v_{r-1}$ and $v_j=u_{q+1}v_{r+3}$. Therefore $v_i$ is adjacent to $u_qv_r$ and $u_{q+2}v_r$ and $d(v_i,u_{q+1}v_{r+1})=2$. Likewise, $v_j$ is adjacent to both $u_qv_{r+2},u_{q+2}v_{r+2}$ and $d(v_j,u_{q+1}v_{r+1})=2$. Thus there exists no vertex $v_l \in V(S''')$ such that $l(v_l)=\alpha_i \in [4]$. This contradicts the fact that $\lambda_1^1(G) \leq 4$, for all $m,n \geq 2.$
\end{proof}
\begin{lem}\label{s} Let $G'\subset G$ with \ \\ $V(G')=\{u_qv_r,u_{q+2}v_r,u_{q+1}v_{r+1},u_qv_{r+2},u_{q+2}v_{r+2},  u_{q+1}v_{r+3},u_qv_{r+4},u_{q+2}v_{r+4}\}$, $q,r \geq 0$. Suppose that $l(u_qv_r), l(u_{q+2}v_{r})$ are $\alpha_0, \alpha_1$ respectively, then $l(u_qv_{r+4}),l(u_{q+2}v_{r+4})$ are both neither $\alpha_0$ nor $\alpha_1$.
\end{lem}
\begin{proof} The vertex set $\left\{u_qv_r,u_{q+2}v_r,u_{q+1}v_{r+1},u_qv_{r+2},u_{q+2}v_{r+2}\right\} \subset V'(G')$ induces a star $S_4 \subset G.$ Since $\lambda_1^1(S_4)=4,$ we have $l(u_{q+1}v_{r+1})=\alpha_2,l(u_qv_{r+2})=\alpha_3,l(u_{q+2}v_{r+2})=\alpha_4$. Set $\left\{u_qv_{r+2},u_{q+2}v_{r+2},u_{q+1}v_{r+3},u_qv_{r+4},u_{q+2}v_{r+4}\right\} \subset V(G')$ induces another star $S'_4 \subset G'$. Clearly, $S_4$ and $S'_4$ are adjacent and $S_4 \cup S'_4=G'$ Now, suppose $l(u_qv_{r+4})=\alpha_0,l(u_{q+2}v_{r+4})=\alpha_1$, or vice versa without the loss of generality. Since $l(u_qv_{r+2})=\alpha_3,$ and $l(u_{q+2}v_{r+2})=\alpha_3$ from the labeling on $S_4,$ the only label left in $[4]$ for $u_{q+1}v_{r+3}$ is $\alpha_2$. This however is a contradiction since $d(u_{q+1}v_{r+1},u_{q+1}v_{r+3})=2$.
\end{proof}
\begin{rem} \label{t} \begin{itemize} \item[(i)] \rm{By theorem \ref{a}, $P_m \times C_n$ is connected if $n$ is odd and not connected if $n$ is even. This is because when $n$ is odd, cycle $C_n$ is non bipartite and when $n$ is even, $C_n$ is bipartite. Now, Let $P_m \times C_n=G=G_1 \cup G_2$, where $n$ is even. Then\\ $V(G_1)=\left\{(u_i,v_j):i \in[(m-1)(\epsilon)],j\in [n(\epsilon)]\; {\rm or}\ i\in [(m-1)(o)], j\in [n(o)]\right\}$ and  \\ $V(G_2)=\left\{(u_i,v_j):i \in[(m-1)(\epsilon)],j\in [n(o)]\; {\rm or}\ i\in [(m-1)(o)], j\in [n(\epsilon)]\right\}$.}
%\end{rem}
%\begin{rem}\label{u}
\item[(ii)] \rm{$G_1 \;{\rm and }\ G_2$ above are isomorphic since $C_n$ is a cycle and they are both components of $G$.}
%\end{rem}
%\begin{rem}\label{v}
\item[(iii)] \rm{Suppose $G=P_m \times C_n$, $n$ odd. Then $G$ is equivalent to $G'$, where $G'$ is one of the two components of $P_m \times C_{2n}$.}
%\end{rem}
%\begin{rem} \label{vi}
\item[(iv)] \rm{$G'$ above is equivalent to the connected component of $P_m \times P_{2n+1}$ such that $u_iv_0$ coincides with $u_iv_{2n}$, for all $i \in [(m-1)(\epsilon)]$ or for all $i\in [(m-1)(o)].$}
\end{itemize}
\end{rem}
\begin{lem} \label{vii} \cite{BBB1}
$\lambda_1^1(C_m)= \left\{
\begin{array}{ll}
            2 &  \mbox{if} \; m\equiv 0 \mod 3    \\
            3 &  m\not\equiv 0 \mod 3; m\neq 5\\
            4 &  m = 5.
\end{array}
\right.$

\end{lem}
\begin{thm}  $\lambda_1^1(P_2 \times C_m)= \left\{
\begin{array}{ll}2 &   \mbox {if} \; m \equiv 0\; \mod \; 3\\
                 3 &   \mbox {otherwise}.
\end{array}
\right. $
\end{thm}
\begin{proof} By  Remark \ref{t}(iii), if $m$ is odd, then $P_2 \times C_m \equiv C_{2m}.$ If $m$ is even, then $P_2 \times C_m$ is a union of $m$-cycles,
 %C_m= C'_m \cup C''_m,$ where
  $C'_m$ and $C''_m$ are $m-cycles$ which are its components.
  %and are components of $P_2 \times C_m.$
  By Lemma \ref{vii}, for $m$ odd, $\lambda_1^1(P_2 \times C_m)=\lambda^1_1(C_{2m})=q$, where $q=2$ for $2m \equiv 0\;\mod \; 3$ and $q=3$ if otherwise. Also $\lambda_1^1(P_2 \times C_2)= \lambda_1^1(C_{n})=p$, where $p=2$ if $n\equiv 0 \mod 3$ and $p=3$  otherwise.
\end{proof}
\begin{thm} \label{viii} For any $m\in \mathbb N$, $m\geq 3,$ $\lambda_1^1(P_m \times C_3)=5$.
\end{thm}
\begin{proof} By Remarks \ref{t}(iii) and (iv), and $P_m \times C_3$ is congruent to a connected component $G'$ of $P_m \times P_7$ with $u_iv_0 \equiv u_iv_6,$ $u_iv_0,u_iv_6 \in V(G')$. Thus, $L(u_iv_0)=L(u_iv_6)$ for all $i\in [(m-i)(\epsilon)].$  Now, let $G''$ be a subgraph of $G'$ induced by the vertex subset \\ $\left\{u_iv_0,u_{i+2}v_0,u_{i+1}v_1,u_{i}v_2,u_{i+2}v_2,u_{i+1}v_3,u_iv_4,u_{i+2}v_4,u_{i+1}v_5,u_iv_6,u_{i+1}v_6\right\} \subseteq V(G')$, for any $i\in [(m-1)(\epsilon)].$ Suppose $\lambda^1_1(G')=4$ and $\alpha_0,\alpha_1,\alpha_2,\alpha_3,\alpha_4 \in [4].$ Let $l(u_iv_0)=\alpha_0$ and $l(u_{i+2}v_0)=\alpha_1$. Then, $l(u_iv_6)=\alpha_0$ and $l(u_{i+2}v_6)=\alpha_1$. Now, suppose $l(u_{i+1}v_1)=\alpha_2$. Since $d(u_{i+1}v_1,u_{i+1}v_5)=2,$ then for some $\alpha_k \in [4]$, $\alpha_k=l(u_{i+1}v_5)\neq \alpha_2$. In fact, $\alpha_k \notin \left\{\alpha_0,\alpha_1,\alpha_2\right\}$. Set $\alpha_k=\alpha_3.$ The vertex subset $\left\{u_iv_0,u_{i+2}v_0,u_{i+1}v_1,u_{i}v_2,u_{i+2}v_2\right\} \subset V(G'')$ induces a star $S_4 \subset G'$ with center $u_{i+1}v_1$. Since $\lambda_1^1(S_4)=4$, if $l(u_iv_2)=\alpha_3,$ then $l(u_{i+2}v_2)=\alpha_4$. Let $A$ and $B$ be vertex subsets of $V(G'),$ such that $A=\left\{u_iv_4,u_{i+2}v_4\right\}$ and $B=\left\{u_iv_2,u_{i+2}v_2,u_{i+1}v_5,u_1v_6,u_{i+2}v_6\right\}$. Clearly, $d(u,v) \leq 2$ for all $u \in A$ and $v \in B$. Then, $l(u_iv_4), l(u_{i+2}v_4)$ $\notin \left\{\alpha_0,\alpha_1,\alpha_2,\alpha_3,\right\}$. Therefore, since $\lambda_1^1(S_4)=4$, $l(u_iv_4)=\alpha_3=l(u_{i+2}v_4).$ But $d(u_iv_4,u_{i+2}v_4)=2$. This a contradiction and hence, $\lambda_1^1(P_m \times C_3) \geq 5.$ \\Claim: Let $\alpha_k \ L(V_i),$ then $\alpha_k \notin V_{i+2}$, for $V_i$,$V_{i+2} \in V(G')$.  \\Reason: For all $v \in V_i$, $u \in V_{i+2},$ $d(u,v) \leq 2.$ \\Now, let $U_i=\left\{u_iv_0,u_iv_2,u_1v_4\right\}$, $U_{i+1}=\left\{u_{i+1}v_1,u_{i+1}v_3,u_{i+1}v_5\right\}$, $U_i,U_{i+1} \subset V(G'').$ $l(u_iv_j)$ labels $u_{i+1}$ for all $v_j,u_k$ in $U_iU_{i+1}$ respectively where $\left|k-j\right|=3$ since $d(u_iv_j,u_iv_k)=3$. Therefore without loss of generality, we say $L(U_i)=L(U_{i+1})=\left\{\alpha_0,\alpha_1,\alpha_2\right\} \subset [5].$ Likewise, let $U_{i+2}=\left\{u_{i+2}v_0,u_{i+2}v_2,u_{i+2}v_4\right\}$ and $U_{i+3}=\left\{u_{i+3}v_1,u_{i+3}v_3,u_{i+3}v_5\right\}$, $U_{i+2},U_{i+3} \subset V(G'').$ $l(u_{i+2}v_l)$ labels $u_{i+3}v_p$ for all $v_l, v_p$ in $U_{i+2},U_{i+3}$ respectively, where $\left|l-p\right|=3$. Thus $L(U_{i+2})=L(U_{i+3})=\left\{\alpha_3,\alpha_4,\alpha_5\right\} \subset [5]$. Based on the last scheme, we have $L(U_a)=L(U_{a+4})$ for any $a \in \left[i,i+3\right]$, where $i \in [(m-1)(\epsilon)].$ Thus there exists a $5-L(1,1)-$labeling of $P_m \times C_3$ and thus $\lambda_1^1(P_m \times C_3) \leq 5$ and then the equality holds.
\end{proof}
\begin{cor} If $m \geq 3$, then, $\lambda_1^1(P_m \times C_6)=5$.
\end{cor}
\begin{proof} Follows from Remark \ref{t} (iii) and Theorem \ref{viii}.
\end{proof}
\begin{thm} If $m \geq 3$, then $\lambda_1^1(P_m \times C_4)=5$.
\end{thm}
\begin{proof} From Remarks \ref{t} (ii) and (iii),  $P_m \times C_4=G_1 \cup G_2$, where $G_1,G_2$ are isomorphic connected components of $P_m \times C_4$. Let $u_iv_0,u_iv_4 \in V(G_1)$, say, for all $i \in [(m-1)(\epsilon)]$, such that $u_iv_0 \equiv u_iv_4$ then by Remark \ref{t}(iv) , $G_1$ is equivalent to a connected component of $P_m \times P_5$. Now, let $G'_1 \subseteq G_1 $ be a subgraph of $G_1$ with \ \\ $V(G'_1)$=$\left\{u_rv_0,u_{r+2}v_0,u_{r+1}v_1,u_rv_2,u_{r+2}v_2,u_{r+1}v_3,u_rv_4,u_{r+2}v_4\right\}$, where $r \leq m-4$. Obviously, $u_rv_0 \equiv u_rv_4 $ and $u_{r+2}v_0 \equiv v_4.$ Thus, $l(u-rv_0)$=$l(u_rv_4)=\alpha_i$ and $l(u_{r+2}v_0)= l(u_{r+2}v_4)=\alpha_j$, $\alpha_i,\alpha_j \in [4]$. By Lemma \ref{s}, there exists a vertex $v \in V(G'_1)$ such that $l(v) \notin [4]$. Thus $\lambda_1^1(G'_1) \geq 5$ and therefore, $\lambda_1^1(G_1) \geq 5$ and finally, $\lambda_1^1(P_m \times C_4) \geq 5.$ Now, for any pair $v_a,v_b \in V(G,_1)$, $d(v_a,v_b)\leq 2.$ Thus $L(V_i) \cap L(V_{i+1})= \emptyset$ and $L(V_i) \cap L(V_{i+2})= \emptyset$. However, $L(V_i)$ labels $L(V_{}i+3)$ since $d(v_a,v_c)=3$ for all $v_a \in V_i$ and $v_c \in V_{i+3}$. Thus, $L(V_i)=L(V_{i+3k})$, $L(V_{i+1})=L(V_{i+4k})$ and $L(V_{i+2})=L(V_{i+5k})$ for all $k \in {\mathbb N}$. since $\left|V(G'_1)\right|=6$, then $\lambda_1^1(P_m \times C_4) \leq 5$ and therefore, the equality follows.
\end{proof}
\begin{thm} If $m \geq 3$, then $\lambda_1^1(P_m \times C_5)=4$.
\end{thm}
\begin{proof} Clearly, $P_m \times C_5 \equiv G_1,$ where $G_1$ is a connected component of $P_m \times C_{10}$. Therefore, $\lambda_1^1(P_m \times C_5) \leq$ $\lambda_1^1(P_m \times C_{10}) \leq$ $\lambda_1^1(C_{10m'} \times C_{10n'})=4$, for all $m',n'\in \mathbb N$. Now, since there exists a star $S_4 \subset P_m \times C_5$, then $\lambda_1^1(P_m \times C_5) \geq 5$.
\end{proof}
 The last theorem clearly yields the next corrolary.
\begin{cor} For all $m \geq 3, n' \in \mathbb N,$ $\lambda_1^1(P_m \times C_{5n'})=4$.
\end{cor}
\begin{lem} \label{1} Suppose $G'$ is a connected component of $P_3 \times P_n$, $n \geq 9$, such that $u_iv_j,u_iv_k \in V(G').$ If $d(u_iv_j,u_iv_k)=8,$ then $l(u_iv_j) \neq l(u_iv_k).$
\end{lem}
\begin{proof} Suppose $\alpha_j,\alpha_k\in [4]$ and $\alpha_j=l(u_1v_j)$, $\alpha_k=l(u_1v_k)$, while $d(u_1v_j,u_1v_k)=8.$ The next vertex, by Lemmas \ref{q} and \ref{r}, that $\alpha_j$ labels is either $u_2v_{j+3}$ and $u_2v_{j+3}$. Now, since $d(u_0v_{j+3},u_1v_k)=5$, then by Lemma \ref{q}, $\alpha_j \neq l(u_1v_k)$. Thus, $\alpha_k \neq \alpha_j$.
\end{proof}
\begin{thm} \label{a1} For $m \geq 3$, $\lambda_1^1(P_m \times C_7)=5$.
\end{thm}
\begin{proof} Suppose $\lambda_1^1(P_m \times C_7)=4$. Clearly from an earlier remark, $P_m \times C_7 \equiv G'$ where $G'$is a connected component of $P_m \times C_{14}$. Also, $G' \equiv G'',$ where $G''$ is the connected component of $P_m \times P_{15}$, with $u_iv_0 \equiv u_iv_{14}$ for all $i\in [(m-1)(\epsilon)]$. Suppose $\bar{G}$ is a subgraph of $G''$ induced by the vertex set $U_i, U_{i+1}$ and $U_{i+2}$ such  that $u_iv_0 \in U_i$, and $u_{i+2}v_0 \in U_{i+2}$. Let $\left\{\alpha_i\right\}_{i=0}^4 =[4]$ and suppose $\alpha_0,\alpha_1,\alpha_2,\alpha_3,\alpha_4,$ labels $u_iv_0, u_{i+2}v_0$ $u_{i+1}v_1,u_{i+2}v_0$ $u_{i+2}v_2$. Then $l(u_0v_{14})= \alpha_0$ and $l(u_2v_{14})= \alpha_1$. Since $d(u_{i+1}v_1,u_{i+2}v_{13})=2$, then $l(u_{i+1}v_{13}) \in\left\{\alpha_3, \alpha_4\right\}.$ Without loss of generality, let $l(u_iv_{13})= \alpha_3$. Then $L(u_0v_{12},u_2v_{12})=\left\{\alpha_2,\alpha_4\right\}$. Now, $d(u_{i+j}v_k,u_{i+1}v_7)=5$ for all $j \in \left\{0,2\right\}$, $k \in \left\{2,12\right\}$. Thus, by Lemma \ref{q}, $l(u_{i+1}v_7) \in A=\left\{\alpha_2, \alpha_3, \alpha_4\right\}$. Also, by the reason of distance, $l(u_{i+1}v_3) \in A$. Thus, $l(u_{i+1}v_3)$ is either $\alpha_0$ or $\alpha_1$. Again without loss of generality, suppose $l(u_{i+1}v_3)= \alpha_0$. By Lemma \ref{r}, $l(u_{i+1}v_7) \neq \alpha_0$. Thus, $l(u_{i+1}v_7) = \alpha_1. $ Since $l(u_{i+1}v_7)= \alpha_1,$ then $l(u_{i+1}v_{11}) \neq \alpha_1.$ Therefore, $l(u_{i+1}v_{11}) \notin \left\{\alpha_1 \cup A\right\}$ and hence, $l(u_{i+1}v_{11})= \alpha_0.$ But contradicts Lemma \ref{1} since $d(u_{i+1}v_3, u_{i+1}v_{11})=8$ and it is assumed that $\lambda_1^1(P_m \times C_7)=4$. Thus, $\lambda_1^1(P_m \times C_7) \geq 5.$ Conversely, for each $i \in [m-1]$, $\left|V_i\right|=7$, where $V_i \subset V(G')$. Therefore, suppose $\left|L(V_i)\right|=6$, then there exists a pair $v_1,v_2 \in V_i$ such that $l(v_i)=l(v_2)=\alpha_k$ for some $\alpha_k \in [5]$. Now, set $u_1=u_iv_j$ and $u_2=u_iv_{j+4}$ such that $d(u_iv_j,v_{j+4})$ $d(u_1,u_2)=4$. Let $\bar{V_1}=V_i \backslash \left\{u_iv_j\right\}$. Set $\alpha_j=l(u_kv_l)=l(u_{k+1}v_{l+3})$ for all u $u_kv_j \in \bar{V_1}$. Now, there exists $u_3=u_{k+3}v_{j+3} \in V_{i+1}$ such that $u_3$ is not yet labeled. Let $u_4=u_{k+1}v_{j-1}$ and set $l(u_{k+1}v_{j-1})=l(u_{k+1}v_{j+3})$. Obviously, $d(u_3,u_4)=4$ and $u_3,u_4 \in V_{i+1}.$ Repeat the above scheme between $V_{i+1}$ and $V_{i+2}$, $V_{i+2}$ and $V_{i+3},...,$ $V_{m-2},V_{m-1}.$ Thus $ \lambda_1^1(P_m \times C_7) \leq 5$ and then the equality follows.
\end{proof}
The proof of the next results follow the last theorem and some remarks made earlier.
\begin{cor}For $m \geq 3$, $\lambda_1^1(P_m \times C_{14})=5.$
\end{cor}
\begin{thm}Let $m \geq 3$. Then $\lambda_1^1(P_m \times C_8)=5$.
\end{thm}
\begin{proof} That $\lambda_1^1(P_m \times C_8) \geq 5$ follows from Lemma \ref{1} and $\lambda_1^1(P_m \times C_8)\leq 5$ follows from repeating the $L(1,1)-$labeling of $P_m \times C_4$.
\end{proof}

{\tiny{
\begin{center}
\pgfdeclarelayer{nodelayer}
\pgfdeclarelayer{edgelayer}
\pgfsetlayers{nodelayer,edgelayer}
\begin{tikzpicture}
%\centering
	\begin{pgfonlayer}{nodelayer}
	
	\node [minimum size=0cm,]  at (-7.5,1.3) {(a) \; $P_4 \times C_{10}$};

		\node [minimum size=0cm,draw,circle] (0) at (-10,2) {$ 0$};
		\node [minimum size=0cm,draw,circle] (1) at (-9,2) {$ 4$};
		\node [minimum size=0cm,draw,circle] (2) at (-8,2) {$ 1$};
		\node [minimum size=0cm,draw,circle] (3) at (-7,2) {$ 3$};
		\node [minimum size=0cm,draw,circle] (4) at (-6,2) {$ 2$};
		\node [minimum size=0cm,draw,circle] (5) at (-5,2) {$ 0$};
		\node [minimum size=0cm,draw,circle] (6) at (-9.5,2.5) {$ 2$};
		\node [minimum size=0cm,draw,circle] (7) at (-8.5,2.5) {$ 0$};
		\node [minimum size=0cm,draw,circle] (8) at (-7.5,2.5) {$ 4$};
		\node [minimum size=0cm,draw,circle] (9) at (-6.5,2.5) {$ 1$};
		\node [minimum size=0cm,draw,circle] (10) at (-5.5,2.5) {$ 3$};
		\node [minimum size=0cm,draw,circle] (11) at (-10,3) {$ 1$};
		\node [minimum size=0cm,draw,circle] (12) at (-9,3) {$ 3$};
		\node [minimum size=0cm,draw,circle] (13) at (-8,3) {$ 2$};
		\node [minimum size=0cm,draw,circle] (14) at (-7,3) {$ 0$};
		\node [minimum size=0cm,draw,circle] (15) at (-6,3) {$ 4$};
		\node [minimum size=0cm,draw,circle] (16) at (-5,3) {$ 1$};
		\node [minimum size=0cm,draw,circle] (17) at (-9.5, 3.5) {$ 4$};
		\node [minimum size=0cm,draw,circle] (18) at (-8.5, 3.5) {$ 1$};
		\node [minimum size=0cm,draw,circle] (19) at (-7.5, 3.5) {$ 3$};
		\node [minimum size=0cm,draw,circle] (20) at (-6.5, 3.5) {$ 2$};
		\node [minimum size=0cm,draw,circle] (21) at (-5.5, 3.5) {$ 0$};
		
		\node [minimum size=0cm,]  at (0,1.3) {(b) \; $P_4 \times C_{12}$};
		\node [minimum size=0cm,draw,circle] (22) at (-3,2) {$ 0$};
		\node [minimum size=0cm,draw,circle] (23) at (-2,2) {$ 2$};
		\node [minimum size=0cm,draw,circle] (24) at (-1,2) {$ 0$};
		\node [minimum size=0cm,draw,circle] (25) at (0,2) {$ 2$};
		\node [minimum size=0cm,draw,circle] (26) at (1,2) {$ 0$};
		\node [minimum size=0cm,draw,circle] (27) at (2,2) {$ 2$};
		\node [minimum size=0cm,draw,circle] (28) at (3,2) {$ 0$};
		\node [minimum size=0cm,draw,circle] (29) at (-2.5,2.5) {$ 4$};
		\node [minimum size=0cm,draw,circle] (30) at (-1.5,2.5) {$ 1$};
		\node [minimum size=0cm,draw,circle] (31) at (-0.5,2.5) {$ 3$};
		\node [minimum size=0cm,draw,circle] (32) at (0.5,2.5) {$ 4$};
		\node [minimum size=0cm,draw,circle] (33) at (1.5,2.5) {$ 1$};
		\node [minimum size=0cm,draw,circle] (34) at (2.5,2.5) {$ 3$};
		\node [minimum size=0cm,draw,circle] (35) at (-3,3) {$ 1$};
		\node [minimum size=0cm,draw,circle] (36) at (-2,3) {$ 3$};
		\node [minimum size=0cm,draw,circle] (37) at (-1,3) {$ 4$};
		\node [minimum size=0cm,draw,circle] (38) at (0,3) {$ 1$};
		\node [minimum size=0cm,draw,circle] (39) at (1,3) {$ 3$};
		\node [minimum size=0cm,draw,circle] (40) at (2,3) {$ 4$};
		\node [minimum size=0cm,draw,circle] (41) at (3,3) {$ 1$};
		\node [minimum size=0cm,draw,circle] (42) at (-2.5,3.5) {$ 2$};
		\node [minimum size=0cm,draw,circle] (43) at (-1.5,3.5) {$ 0$};
		\node [minimum size=0cm,draw,circle] (44) at (-0.5,3.5) {$ 2$};
		\node [minimum size=0cm,draw,circle] (45) at (0.5,3.5) {$ 0$};
		\node [minimum size=0cm,draw,circle] (46) at (1.5,3.5) {$ 2$};
		\node [minimum size=0cm,draw,circle] (47) at (2.5,3.5) {$ 0$};
	\end{pgfonlayer}
	\begin{pgfonlayer} {edgelayer}
		\draw [thin=1.00] (0) to (6);
		\draw [thin=1.00] (1) to (6);
		\draw [thin=1.00] (1) to (7);
		\draw [thin=1.00] (2) to (7);
		\draw [thin=1.00] (2) to (8);
		\draw [thin=1.00] (3) to (8);
		\draw [thin=1.00] (3) to (9);
		\draw [thin=1.00] (4) to (9);
		\draw [thin=1.00] (4) to (10);
		\draw [thin=1.00] (5) to (10);
		\draw [thin=1.00] (6) to (11);
		\draw [thin=1.00] (6) to (12);
		\draw [thin=1.00] (7) to (12);
		\draw [thin=1.00] (7) to (13);
		\draw [thin=1.00] (8) to (13);
		\draw [thin=1.00] (8) to (14);
		\draw [thin=1.00] (9) to (14);
		\draw [thin=1.00] (9) to (15);
		\draw [thin=1.00] (10) to (15);
		\draw [thin=1.00] (10) to (16);
		\draw [thin=1.00] (11) to (17);
		\draw [thin=1.00] (12) to (17);
		\draw [thin=1.00] (12) to (18);
		\draw [thin=1.00] (13) to (18);
		\draw [thin=1.00] (13) to (19);
		\draw [thin=1.00] (14) to (19);
		\draw [thin=1.00] (14) to (20);
		\draw [thin=1.00] (15) to (20);
		\draw [thin=1.00] (15) to (21);
		\draw [thin=1.00] (16) to (21);
		
    \draw [thin=1.00] (22) to (29);
		\draw [thin=1.00] (23) to (29);
		\draw [thin=1.00] (23) to (30);
		\draw [thin=1.00] (24) to (30);
		\draw [thin=1.00] (24) to (31);
		\draw [thin=1.00] (25) to (31);
		\draw [thin=1.00] (25) to (32);
		\draw [thin=1.00] (26) to (32);
		\draw [thin=1.00] (26) to (33);
		
		\draw [thin=1.00] (27) to (33);
		\draw [thin=1.00] (27) to (34);
		\draw [thin=1.00] (28) to (34);
		
		\draw [thin=1.00] (29) to (35);
		\draw [thin=1.00] (29) to (36);
		\draw [thin=1.00] (30) to (36);
		\draw [thin=1.00] (30) to (37);
		\draw [thin=1.00] (31) to (37);
		\draw [thin=1.00] (31) to (38);
		\draw [thin=1.00] (32) to (38);
		\draw [thin=1.00] (32) to (39);
		\draw [thin=1.00] (33) to (39);
		\draw [thin=1.00] (33) to (40);
		\draw [thin=1.00] (34) to (40);
		\draw [thin=1.00] (34) to (41);
	
		\draw [thin=1.00] (35) to (42);
		\draw [thin=1.00] (36) to (42);		
		\draw [thin=1.00] (36) to (43);
		\draw [thin=1.00] (37) to (43);
		\draw [thin=1.00] (37) to (44);
		\draw [thin=1.00] (38) to (44);
		\draw [thin=1.00] (38) to (45);
		\draw [thin=1.00] (39) to (45);
		\draw [thin=1.00] (39) to (46);
		\draw [thin=1.00] (40) to (46);
		\draw [thin=1.00] (40) to (47);
		\draw [thin=1.00] (41) to (47);

	\end{pgfonlayer}
\end{tikzpicture}
\end{center}

\begin{center}
\pgfdeclarelayer{nodelayer}
\pgfdeclarelayer{edgelayer}
\pgfsetlayers{nodelayer,edgelayer}
\begin{tikzpicture}
%\centering
	\begin{pgfonlayer} {nodelayer}
	
		\node [minimum size=0cm,]  at (0,1.3) {(c) $P_4 \times C_{16}$};
		\node [minimum size=0cm,draw,circle] (a) at (-4,2) {$ 0$};
		\node [minimum size=0cm,draw,circle] (22) at (-3,2) {$ 4$};
		\node [minimum size=0cm,draw,circle] (23) at (-2,2) {$ 1$};
		\node [minimum size=0cm,draw,circle] (24) at (-1,2) {$ 4$};
		\node [minimum size=0cm,draw,circle] (25) at (0,2) {$ 2$};
		\node [minimum size=0cm,draw,circle] (26) at (1,2) {$ 4$};
		\node [minimum size=0cm,draw,circle] (27) at (2,2) {$ 3$};
		\node [minimum size=0cm,draw,circle] (28) at (3,2) {$ 4$};
		\node [minimum size=0cm,draw,circle] (b) at (4,2) {$ 0$};
		
		\node [minimum size=0cm,draw,circle] (c) at (-3.5,2.5) {$ 2$};
		\node [minimum size=0cm,draw,circle] (29) at (-2.5,2.5) {$ 0$};
		\node [minimum size=0cm,draw,circle] (30) at (-1.5,2.5) {$ 3$};
		\node [minimum size=0cm,draw,circle] (31) at (-0.5,2.5) {$ 1$};
		\node [minimum size=0cm,draw,circle] (32) at (0.5,2.5) {$ 0$};
		\node [minimum size=0cm,draw,circle] (33) at (1.5,2.5) {$ 2$};
		\node [minimum size=0cm,draw,circle] (34) at (2.5,2.5) {$ 1$};
		\node [minimum size=0cm,draw,circle] (d) at (3.5,2.5) {$ 3$};

		\node [minimum size=0cm,draw,circle] (e) at (-4,3) {$ 1$};
		\node [minimum size=0cm,draw,circle] (35) at (-3,3) {$ 3$};
		\node [minimum size=0cm,draw,circle] (36) at (-2,3) {$ 2$};
		\node [minimum size=0cm,draw,circle] (37) at (-1,3) {$ 0$};
		\node [minimum size=0cm,draw,circle] (38) at (0,3) {$ 3$};
		\node [minimum size=0cm,draw,circle] (39) at (1,3) {$ 1$};
		\node [minimum size=0cm,draw,circle] (40) at (2,3) {$ 0$};
		\node [minimum size=0cm,draw,circle] (41) at (3,3) {$ 2$};
		\node [minimum size=0cm,draw,circle] (f) at (4,3) {$ 1$};
		
		\node [minimum size=0cm,draw,circle] (g) at (-3.5,3.5) {$ 4$};
		\node [minimum size=0cm,draw,circle] (42) at (-2.5,3.5) {$ 1$};
		\node [minimum size=0cm,draw,circle] (43) at (-1.5,3.5) {$ 4$};
		\node [minimum size=0cm,draw,circle] (44) at (-0.5,3.5) {$ 2$};
		\node [minimum size=0cm,draw,circle] (45) at (0.5,3.5) {$ 2$};
		\node [minimum size=0cm,draw,circle] (46) at (1.5,3.5) {$ 3$};
		\node [minimum size=0cm,draw,circle] (47) at (2.5,3.5) {$ 4$};
		\node [minimum size=0cm,draw,circle] (h) at (3.5,3.5) {$ 0$};
	\end{pgfonlayer}
	\begin{pgfonlayer} {edgelayer}

   \draw [thin=1.00] (a) to (c);
   \draw [thin=1.00] (22) to (c);
	  \draw [thin=1.00] (22) to (29);
		\draw [thin=1.00] (23) to (29);
		\draw [thin=1.00] (23) to (30);
		\draw [thin=1.00] (24) to (30);
		\draw [thin=1.00] (24) to (31);
		\draw [thin=1.00] (25) to (31);
		\draw [thin=1.00] (25) to (32);
		\draw [thin=1.00] (26) to (32);
		\draw [thin=1.00] (26) to (33);
		\draw [thin=1.00] (27) to (33);
		\draw [thin=1.00] (27) to (34);
		\draw [thin=1.00] (28) to (34);
		\draw [thin=1.00] (28) to (d);
   \draw [thin=1.00] (b) to (d);
   \draw [thin=1.00] (c) to (e);
   \draw [thin=1.00] (c) to (35);
		\draw [thin=1.00] (29) to (35);
		\draw [thin=1.00] (29) to (36);
		\draw [thin=1.00] (30) to (36);
		\draw [thin=1.00] (30) to (37);
		\draw [thin=1.00] (31) to (37);
		\draw [thin=1.00] (31) to (38);
		\draw [thin=1.00] (32) to (38);
		\draw [thin=1.00] (32) to (39);
		\draw [thin=1.00] (33) to (39);
		\draw [thin=1.00] (33) to (40);
	\draw [thin=1.00] (d) to (41);
   \draw [thin=1.00] (d) to (f);
   \draw [thin=1.00] (e) to (g);
   \draw [thin=1.00] (35) to (g);
		\draw [thin=1.00] (34) to (40);
		\draw [thin=1.00] (34) to (41);
	  \draw [thin=1.00] (35) to (42);
		\draw [thin=1.00] (36) to (42);		
		\draw [thin=1.00] (36) to (43);
		\draw [thin=1.00] (37) to (43);
		\draw [thin=1.00] (37) to (44);
		\draw [thin=1.00] (38) to (44);
		\draw [thin=1.00] (38) to (45);
		\draw [thin=1.00] (39) to (45);
		\draw [thin=1.00] (39) to (46);
		\draw [thin=1.00] (40) to (46);
		\draw [thin=1.00] (40) to (47);
		\draw [thin=1.00] (41) to (47);
	\draw [thin=1.00] (41) to (h);
   \draw [thin=1.00] (f) to (h);

		\end{pgfonlayer}
\end{tikzpicture}
\end{center}

\begin{center}
\pgfdeclarelayer{nodelayer}
\pgfdeclarelayer{edgelayer}
\pgfsetlayers{nodelayer,edgelayer}
\begin{tikzpicture}
	
\begin{pgfonlayer}{nodelayer}
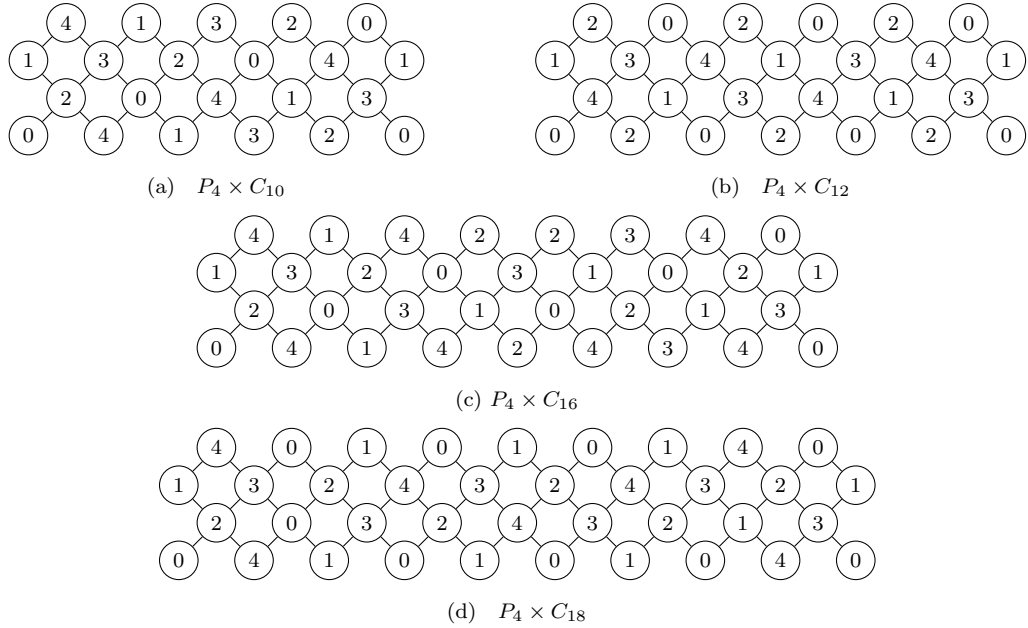
	  \node [minimum size=0cm,]  at (-0.5,0.3) { \; Fig. 1 \;  $4-L(1,1)-$Labeling of $P_4 \times C_n, \; n=10,12,16,18$};
		\node [minimum size=0cm,]  at (-0.5,1.3) {(d) \; $P_4 \times C_{18}$};
		\node [minimum size=0cm,draw,circle] (i) at (-5,2) {$ 0$};
		\node [minimum size=0cm,draw,circle] (a) at (-4,2) {$ 4$};
		\node [minimum size=0cm,draw,circle] (22) at (-3,2) {$ 1$};
		\node [minimum size=0cm,draw,circle] (23) at (-2,2) {$ 0$};
		\node [minimum size=0cm,draw,circle] (24) at (-1,2) {$ 1$};
		\node [minimum size=0cm,draw,circle] (25) at (0,2) {$ 0$};
		\node [minimum size=0cm,draw,circle] (26) at (1,2) {$ 1$};
		\node [minimum size=0cm,draw,circle] (27) at (2,2) {$ 0$};
		\node [minimum size=0cm,draw,circle] (28) at (3,2) {$ 4$};
		\node [minimum size=0cm,draw,circle] (b) at (4,2) {$ 0$};

		\node [minimum size=0cm,draw,circle] (j) at (-4.5,2.5) {$ 2$};
		\node [minimum size=0cm,draw,circle] (c) at (-3.5,2.5) {$ 0$};
		\node [minimum size=0cm,draw,circle] (29) at (-2.5,2.5) {$ 3$};
		\node [minimum size=0cm,draw,circle] (30) at (-1.5,2.5) {$ 2$};
		\node [minimum size=0cm,draw,circle] (31) at (-0.5,2.5) {$ 4$};
		\node [minimum size=0cm,draw,circle] (32) at (0.5,2.5) {$ 3$};
		\node [minimum size=0cm,draw,circle] (33) at (1.5,2.5) {$ 2$};
		\node [minimum size=0cm,draw,circle] (34) at (2.5,2.5) {$ 1$};
		\node [minimum size=0cm,draw,circle] (d) at (3.5,2.5) {$ 3$};
		
		\node [minimum size=0cm,draw,circle] (k) at (-5,3) {$ 1$};
		\node [minimum size=0cm,draw,circle] (e) at (-4,3) {$ 3$};
		\node [minimum size=0cm,draw,circle] (35) at (-3,3) {$ 2$};
		\node [minimum size=0cm,draw,circle] (36) at (-2,3) {$ 4$};
		\node [minimum size=0cm,draw,circle] (37) at (-1,3) {$ 3$};
		\node [minimum size=0cm,draw,circle] (38) at (0,3) {$ 2$};
		\node [minimum size=0cm,draw,circle] (39) at (1,3) {$ 4$};
		\node [minimum size=0cm,draw,circle] (40) at (2,3) {$ 3$};
		\node [minimum size=0cm,draw,circle] (41) at (3,3) {$ 2$};
		\node [minimum size=0cm,draw,circle] (f) at (4,3) {$ 1$};
		
		\node [minimum size=0cm,draw,circle] (l) at (-4.5,3.5) {$ 4$};
		\node [minimum size=0cm,draw,circle] (g) at (-3.5,3.5) {$ 0$};
		\node [minimum size=0cm,draw,circle] (42) at (-2.5,3.5) {$ 1$};
		\node [minimum size=0cm,draw,circle] (43) at (-1.5,3.5) {$ 0$};
		\node [minimum size=0cm,draw,circle] (44) at (-0.5,3.5) {$ 1$};
		\node [minimum size=0cm,draw,circle] (45) at (0.5,3.5) {$ 0$};
		\node [minimum size=0cm,draw,circle] (46) at (1.5,3.5) {$ 1$};
		\node [minimum size=0cm,draw,circle] (47) at (2.5,3.5) {$ 4$};
		\node [minimum size=0cm,draw,circle] (h) at (3.5,3.5) {$ 0$};
	\end{pgfonlayer}
	\begin{pgfonlayer}{edgelayer}

   \draw [thin=1.00] (i) to (j);
   \draw [thin=1.00] (a) to (j);
   \draw [thin=1.00] (a) to (c);
   \draw [thin=1.00] (22) to (c);
	  \draw [thin=1.00] (22) to (29);
		\draw [thin=1.00] (23) to (29);
		\draw [thin=1.00] (23) to (30);
		\draw [thin=1.00] (24) to (30);
		\draw [thin=1.00] (24) to (31);
		\draw [thin=1.00] (25) to (31);
		\draw [thin=1.00] (25) to (32);
		\draw [thin=1.00] (26) to (32);
		\draw [thin=1.00] (26) to (33);
		\draw [thin=1.00] (27) to (33);
		\draw [thin=1.00] (27) to (34);
		\draw [thin=1.00] (28) to (34);
		\draw [thin=1.00] (28) to (d);
   \draw [thin=1.00] (b) to (d);
   \draw [thin=1.00] (j) to (k);
   \draw [thin=1.00] (j) to (e);
   \draw [thin=1.00] (c) to (e);
   \draw [thin=1.00] (c) to (35);
		\draw [thin=1.00] (29) to (35);
		\draw [thin=1.00] (29) to (36);
		\draw [thin=1.00] (30) to (36);
		\draw [thin=1.00] (30) to (37);
		\draw [thin=1.00] (31) to (37);
		\draw [thin=1.00] (31) to (38);
		\draw [thin=1.00] (32) to (38);
		\draw [thin=1.00] (32) to (39);
		\draw [thin=1.00] (33) to (39);
		\draw [thin=1.00] (33) to (40);
	\draw [thin=1.00] (d) to (41);
   \draw [thin=1.00] (d) to (f);
   \draw [thin=1.00] (k) to (l);
   \draw [thin=1.00] (e) to (l);
   \draw [thin=1.00] (e) to (g);
   \draw [thin=1.00] (35) to (g);
		\draw [thin=1.00] (34) to (40);
		\draw [thin=1.00] (34) to (41);
	  \draw [thin=1.00] (35) to (42);
		\draw [thin=1.00] (36) to (42);		
		\draw [thin=1.00] (36) to (43);
		\draw [thin=1.00] (37) to (43);
		\draw [thin=1.00] (37) to (44);
		\draw [thin=1.00] (38) to (44);
		\draw [thin=1.00] (38) to (45);
		\draw [thin=1.00] (39) to (45);
		\draw [thin=1.00] (39) to (46);
		\draw [thin=1.00] (40) to (46);
		\draw [thin=1.00] (40) to (47);
		\draw [thin=1.00] (41) to (47);
	\draw [thin=1.00] (41) to (h);
   \draw [thin=1.00] (f) to (h);

		\end{pgfonlayer} %{edgelayer}
\end{tikzpicture}
\end{center}
}}

\begin{thm}Given that $n \geq 9$, $n \neq 14$, then $\lambda_1^1(P_4 \times C_n)=4.$
\end{thm}
\begin{proof}
From  (b),(c),(d) if Fig. 1, we notice that $\lambda_1^1(P_4 \times C_{n'})=4$, for all $m'\in \left\{12,16,18\right\}$. Now, by combining each of (b),(c),(d) with (a), we see that $\lambda_1^1(P_m \times C_{n'+10})=4$, for each $n'\in \left\{12,16,18\right\}$. Therefore, $\lambda_1^1(P_4 \times C_{km'+p})=4$ $\forall k\geq 0$ and $p \in \left\{0,10\right\}.$ Thus by an earlier remark, $\lambda_1^1(P_4 \times C_n)=4$ for all $n \geq 9$, $n \neq 14$.
\end{proof}
\begin{cor} Given that $n \geq 9$, $n \neq 14$, and that $m \in \left\{3,4\right\}$ then $\lambda_1^1(P_m \times C_n)=4.$
\end{cor}
\begin{thm}For $m \geq 3$, $\lambda_1^1(P_m \times C_{14})=4$.
\end{thm}
\begin{proof} It follows directly from Remark \ref{t}(iii) and Theorem \ref{a1}.
\end{proof}
 Next, we derive the general lower bound for the $L(1,1)-$ labelling of $P_m \times C_n,$ where $m \geq 5,$ $n \not\equiv 0 \mod 5$. That $\lambda_1^1(P_m \times C_n)=4$, where $m,n$ are both multiples of $5$, has already been established. We need the next lemma to prove the theorem that follows.
\begin{lem} \label{a2}If $\lambda_1^1(P_m \times C_n)=4$ for $n \not\equiv 0 \mod 5$, $n \geq 9$. Then, for all $V_j \subset V(P_m \times C_n),$ $0 \leq j \leq n-2$, there exist $v_a,v_b \in V_j,$ such that $l(v_a)=l(v_b)$ and $d(v_a,v_b)=6$.
\end{lem}
\begin{proof} Let $G=P_m \times C_n$. %By Remark \ref{t}(iii), either $n$ or $2n$ is even.
Suppose, without loss of generality, that $n$ is even since by Remark \ref{t}(iii), if $n$ is odd then G is equivalent to one of the two components of $P_m \times C_{2n} $. Let $G'$ be the connected component of $G$. Let $V'_j \subset V(G')$ such that $V'_j \subset V_j.$ Let $v_a \in V'_j$ such that $l(v_a)=\alpha_k \in [4]$. Since $n$ is not a multiple of $5$, and $n \geq 9$, then $\left|V'_j\right|= \frac{n}{2} > 5$. Since $\lambda_1^1(G)=4,$ then there exists at least some vertex $v_b \in V'_j$ such that $l(v_b)= \alpha_k.$ By the definition of $L(1,1)-$ labeling, $d(v_a,v_b)\neq 2$. Likewise by Lemmas \ref{r} and \ref{1}, $d(v_a, v_b) \notin \left\{4,8\right\} $. Thus, $d(v_a,v_b)=6$.
\end{proof}
\begin{thm} \label{a3} Let $m \geq 5$, $n \not\equiv 0\mod  5$ and $n \geq 9.$ Then, $\lambda_1^1 (P_m \times C_n) \geq 5.$
\end{thm}
\begin{proof} Let $m \geq 5, \; n \not \equiv 0 \mod 5$ and $n \geq 9$. Suppose $\lambda^1_1(P_m \times C_n)=4$. Let $G=P_m \times C_n.$ Suppose $n$ is even. Then there exists $G'$, a connected component of $P_m \times C_n$. (If $n$ is odd, we know from an earlier result that $G$ is a connected component of $P_m \times C_{2n}$.) We defined an arbitrary vertex set $V(G'')=\left\{u_iv_j,u_iv_{j+2},u_{i+1}v_{j+1},u_{i+2}v_j,u_{i+2}v_{j+2},u_{i+3}v_{j+1},u_{i+4}v_j,u_{i+4}v_{j+2}\right\}$, with $V(G'')$ $\subset$ $V(G')$. Clearly, $V(G'')$ induces a subgraph $G''$ of $G$'such that $G''=S'_4 \cup S''_4$ where $S'_4,S''_4$ are stars with $V(S'_4)=\left\{ u_iv_j,u_iv_{j+2},u_{i+1}v_{j+1},u_{i+2}v_j,u_{i+2}v_{j+2},\right\}$ and \\$S''_4=u_{i+2}v_j,u_{i+2}v_{j+2},u_{i+3}v_{j+i},u_{i+4}v_j,u_{i+4}v_{j+2}$ respectively. Now, by \ref{a2} above, for all $V_i \subset V(G'), 0 \leq i \leq m-2$ there exist at least a vertex pair $v_a,v_b \in V_i$ such that for some $\alpha_i \in L(V_i) \subseteq [4],$ $l(v_a)=l(v_b)=\alpha_i$ and $d(v_a,v_b)=6.$ Suppose $u_{i+2}v_{j-2},u_{i+2}v_{j+4} \in V_{i+2}$ such that $l(u_{i+2}v_{j-2})=l(u_{i+2}v_{j+4})=\alpha_i$. There exist vertices $u_{i+1}v_{j+1} \in V_{i+1}$ and $u_{i+3}v_{j+1} \in V_{i+3}$. By Lemma  \ref{q}, $l(u_{i+1}v_{j+1})=\alpha_i$ or $l(u_{i+3}v_{j+1})= \alpha_i$. Suppose $l(u_{i+1}v_{j+1})= \alpha_i$, then $d(u_a,u_b) \leq 2$ for any $u_a \in V(S''_4)$ and $u_b \in \left\{u_{i+1}v_{j+1},u_{i+2}v_{j-2},u_{i+2}v_{j+4}\right\}.$ Thus there is no such vertex as $u_a \in S'_4$ such that $l(u_a)= \alpha_i \in V(S'_4)$. Likewise, $d(u'_a,u_b) \leq 2$ for any $u'_a \in V(S'_5)$ and $u_b \in \left\{u_{i+3}v_{j+1},u_{i+2}v_{j-2},u_{i+2}v_{j+4}\right\}$. Thus, there exists no vertex $u'_a \in V(S'_5),$ such that $l(u'_a)= \alpha_j \in [4]$ and therefore, a contradiction.
\end{proof}
By the result obtained in Theorem \ref{a3}, we see that the $\lambda_1^1(P_m \times C_n) \geq 5$ for all $m \geq 5$ and $n \geq 9,$ where $n$ is not a multiple of 5. In the subsequent results, we obtain the $ \lambda_1^1-$number for the remaining $P_m \times C_n$ graphs.

\newpage

{\tiny
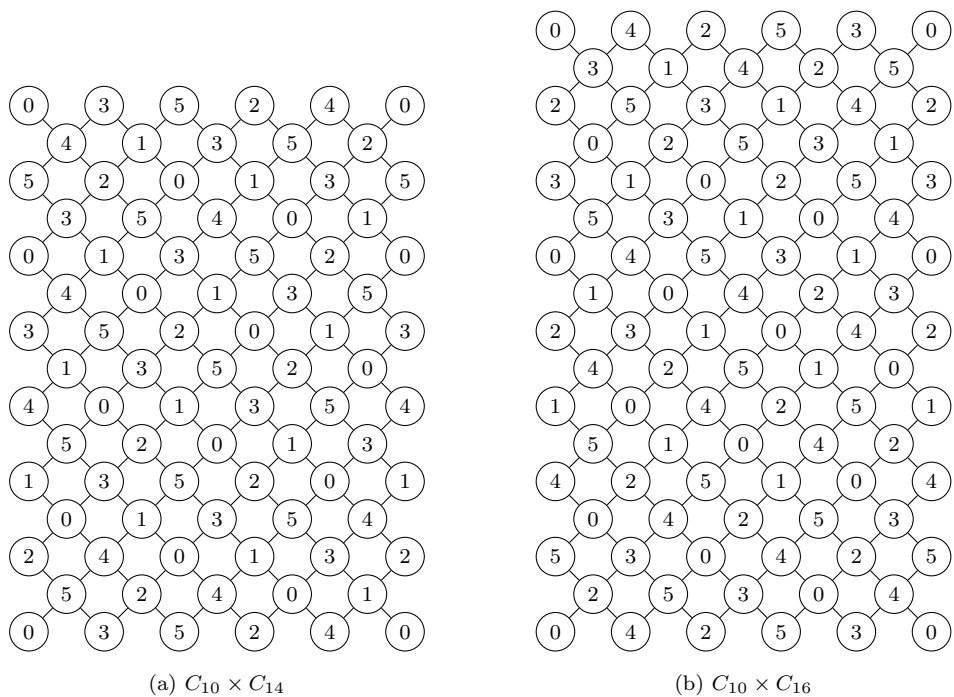
{
\begin{center}
\pgfdeclarelayer{nodelayer}
\pgfdeclarelayer{edgelayer}
\pgfsetlayers{nodelayer,edgelayer}

\end{center}
}}
%\newpage
%We need the following definition for the next results. Let $m'$ be any positive integers, $n'$ be a non-negative integer and let $A=\left\{12,14,16,18\right\}$.
\begin{thm} Let $k \in A$. For all $k,m',n'$, $\lambda_1^1(C_{10m'} \times C_{k+10n'})=5$, where $m'$ is any positive integer, $n'$  a non-negative integer and $A=\left\{12,14,16,18\right\}$.
\end{thm}
\begin{proof} The result follows by combining the $5-$labeling of $C_{10} \times C_{10n'}$ which is obtainable from $n'-$times repeat of Fig.5 a, with the $5-$labeling of $C_{10} \times C_{12}$, $C_{10} \times C_{14}$, $C_{10} \times C_{16}$ and $C_{10} \times C_{18}$ in Fig.5 b  and  of Fig.3 a, b and c respectively along with $C_n$ and then $m'$-copy the resultant graph along with $C_m$.
\end{proof}
\begin{cor} For all $P_m \times C_n$, where $m \geq 5$ and $n \geq 6$, $n \not\equiv 0 \mod 5$ then $\lambda_1^1(P_m \times C_m)=5.$
\end{cor}
\begin{proof}Let $h$ be a positive even integer with $h \geq 12$. Let $k \in A=\left\{12,14,16,18\right\}$. Then, for all $h,\; h \equiv 0 \mod  k+10n^{\prime} $ for some $k \in A.$ The result thus follows from Remarks \ref{t} (iii) and (iv) and the fact that $P_m \times C_n \subset P_{10m} \times C_n$.
\end{proof}
\section{Conclusion}
The following summarizes the results obtained in this work:

For $G=P_m \times P_n$:
\begin{center}
 $\begin{array}{|c|c|c|} \hline
 m & n & \lambda_1^1(P_m \times P_n)\\ \hline
 2 & 2 & 1\\ \hline
 \geq 3 & 2 & 2\\ \hline
 \geq 3  & \geq 3 & 4\\ \hline

 \end{array}$
 \end{center}
 For $G=P_m \times C_n$:
\begin{center}
 $\begin{array}{|c|c|c|} \hline
 m & n & \lambda_1^1(P_m \times C_n)\\ \hline
 2 & \equiv 0 \mod 3 & 2\\ \hline
 2 & \rm \not\equiv 0 \mod 3 & 3\\ \hline
 \geq 3  & \in \left\{3,4,6,7,8,14\right\} & 5\\ \hline
 \geq 3 & \equiv 0 \mod 5 & 4\\ \hline
  3,4 &  \geq 9,  \neq 14  & 4\\ \hline
 \geq 5  & \geq 9,  \not \equiv 0 \mod 5 & 5\\ \hline
 \end{array}$
 \end{center}
%\section{References}

\end{document}